\documentclass[11pt,reqno]{amsart}
\usepackage{amssymb,mathrsfs,bm,enumerate}
\usepackage{amsfonts,dsfont,
  mathptmx,amsaddr}

\makeatletter

\newtheorem{theorem}{Theorem}[section]

\newtheorem{proposition}[theorem]{Proposition}
\theoremstyle{definition} \newtheorem{remark}[theorem]{Remark}

\numberwithin{equation}{section}

\def\Ric{{\operatorname{Ric}}} 
 
\def\gap{{\operatorname{gap}}}

\newcommand\1{\hbox{\kern.375em\vrule height1.57ex depth-.1ex
    width.05em\kern-.375em \rm 1}}

 \newcommand\E{\mathbb{E}}
 
 \newcommand\R{\mathbb{R}}
\renewcommand\P{\mathbb{P}}

\def\mathpal#1{\mathop{\mathchoice{\text{\rm #1}}%
    {\text{\rm #1}}{\text{\rm #1}}%
    {\text{\rm #1}}}\nolimits} \def\id{{\mathpal{id}}}
\def\OM{\mathpal{O}(M)} 
\def\bolddot{{\displaystyle\boldsymbol{\cdot}}}

\def\partr#1#2{/\!/_{\!#1,#2}^{\phantom{.}}}
\def\partrinv#1#2{/\!/_{\!#1,#2}^{-1}}

 \def\vd{\mathrm{d}} \def\r{\right}
\def\l{\left} \def\e{\operatorname{e}}

  \def\mathpal#1{\mathop{\mathchoice{\text{\rm #1}}%
      {\text{\rm #1}}{\text{\rm #1}}%
      {\text{\rm #1}}}\nolimits} \def\id{{\mathpal{id}}}

  \makeatother

\begin{document}

  \title[Spectral gap on Riemannian path space] {Spectral gap on
    Riemannian path space over static and \\ evolving manifolds}

  \author{Li-Juan Cheng\textsuperscript{1,2} and Anton
    Thalmaier\textsuperscript{1}}

  \address{\textsuperscript{1}Mathematics Research Unit, FSTC, University of Luxembourg\\
    6, rue Richard Coudenhove-Kalergi, 1359 Luxembourg, Grand Duchy of
    Luxembourg}
  \address{\textsuperscript{2}Department of Applied Mathematics, Zhejiang University of Technology\\
    Hangzhou 310023, The People's Republic of China}
  \email{lijuan.cheng@uni.lu \text{and} chenglj@zjut.edu.cn}
  \email{anton.thalmaier@uni.lu}

\begin{abstract}
  In this article, we continue the discussion of Fang-Wu (2015) to
  estimate the spectral gap of the Ornstein-Uhlenbeck operator on path
  space over a Riemannian manifold of pinched Ricci curvature. Along
  with explicit estimates we study the short-time asymptotics of the
  spectral gap. The results are then extended to the path space of
  Riemannian manifolds evolving under a geometric flow.  Our paper is
  strongly motivated by Naber's recent work (2015) on characterizing
  bounded Ricci curvature through stochastic analysis on path space.
\end{abstract}

\date{\today}

\maketitle

\section{Introduction}\label{Sect:1}

Suppose $(M,g)$ is a $d$-dimensional complete smooth manifold.
Let $\nabla$ be the Levi-Civita connection and $\Delta$ the Laplacian
on $M$.  Given a $C^1$ vector field $Z$ on $M$, we consider the
Bakry-Emery curvature $\Ric^Z:=\Ric-\nabla Z$ for the so-called Witten
Laplacian $L=\frac{1}{2}(\Delta+Z)$ where $\Ric$ is the Ricci
curvature tensor with respect to $g$.

It is well known that the spectral gap of $L$ can be estimated in
terms of a lower curvature bound $K$, i.e., $$\Ric^Z\geq K$$ for some
constant $K$, see e.g. \cite{CW94,CW97,CGY}.  These results reveal the
close relationship between spectral gap, convergence to equilibrium
and hypercontractivity of the corresponding semigroup. For example,
Poincar\'{e} inequalities and log-Sobolev inequalities which can be
used to characterize the convergence for the semigroup, imply certain
lower bound for the spectral gap.

In this article, we extend this circle of ideas to the Riemannian path
space over $M$ and revisit the problem of estimating the spectral gap
of the Ornstein-Uhlenbeck operator under the following general
curvature condition: there exist constants $k_1$ and $k_2$ such that
$$k_1\leq \Ric^Z\leq k_2.$$

Before moving on, let us briefly summarize some background results on
stochastic analysis on path space over a Riemannian manifold.
Stochastic analysis on path space attracted a lot of attention since
1992 when B.K.~Driver \cite{D92} proved quasi-invariance of the Wiener
measure on the path space over a compact Riemannian manifolds.  A
milestone in the theory is the integration by parts formula (see
e.g. \cite{Bismut,FM93}) for the associated gradient operator induced
by the quasi-invariant flow. This result is a main tool in proving
functional inequalities for the corresponding Dirichlet form, for
instance, the log-Sobolev inequality \cite{AE}; the constant in this
inequality has been estimated in \cite{H97} in terms of curvature
bounds.

Very recently, A.~Naber \cite{Naber} proved that certain log-Sobolev
inequalities and $L^{p}$-inequalities on path space are equivalent to
a constant upper bound for the norm of Ricci curvature on the base
manifold $M$; R.~Haslhofer and A.~Naber \cite{HN} extended these
results to characterize solutions of the Ricci flow, see also
\cite{HN1}.  Inspired by this work, S.~Fang and B.~Wu \cite{FW15} gave
an estimate of the spectral gap under the curvature condition that
there exist two constants $k_1$ and $k_2$ with $k_1+k_2\geq 0$ such
that
$$k_1\leq \Ric^Z\leq k_2.$$
Dealing with the case ``$k_1+k_2< 0$'' however, the same argument may
lead to a loss of information concerning $k_2$. We revisit this topic
in this article.  Our aim is to remove the restriction $k_1+k_2\geq 0$
in the curvature condition and to establish sharper short-time
asymptotics for the spectral gap.

Our methods rely strongly on suitable extensions and generalizations
of recent estimates on Riemannian path space, due to Naber
\cite{Naber}, resp.~ Haslhofer and Naber \cite{HN,HN1}.  This work is
crucial for our arguments, as it allows to characterize bounded Ricci
curvature in terms of stochastic analysis on path space.

We start by briefly introducing the context.  Let $X_t^x$ be a
diffusion process with generator $L$ starting from $X^x_0=x$. We call
$X_t^x$ an $L$-diffusion process.  We assume that $X_t^x$ is
non-explosive.  Let $B_t=(B_t^1,\ldots,B_t^d)$ be a $\R^d$-valued
Brownian motion on a complete filtered probability space
$(\Omega, \{\mathscr{F}_t\}_{t\geq 0},\mathbb{P})$ with the natural
filtration $\{\mathscr{F}_t\}_{t\geq 0}$.  It is well known that the
$L$-diffusion process $X^x_t$ starting from $x$ solves the equation
\begin{align}\label{SDE-X}
  \vd X^x_t=u^x_t\circ \vd B_t+\frac{1}{2}Z(X^x_t)\,\vd
  t,\quad X_0^x=x,
\end{align}
where $u^x_t$ is the horizontal process of $X^x_t$ taking values in
the orthonormal frame bundle $\OM$ over $M$ such that $\pi(u_0^x)=x$.
Furthermore
$$\partr st:=u^x_t\circ (u^x_s)^{-1}\colon{T_{X^x_s}M}\to{T_{X^x_t}M},\quad s\leq t,$$
denotes parallel transport along the paths $r\mapsto X^x_r$.  As
usual, orthonormal frames $u\in\OM$ are identified with isometries
$u\colon\R^d\to T_xM$ where $\pi(u)=x$.

For fixed $T>0$ define $W^T:=C([0,T]; M)$ and let
$$\mathscr{F}C_{0,T}^{\infty}=\left\{W^T\ni\gamma \mapsto f(\gamma_{t_1},\ldots, \gamma_{t_n})\colon\
  n\geq 1,\ 0<t_1<\ldots < t_n \leq T,\ f \in
  C_0^{\infty}(M^n)\right\}$$
be the class of smooth cylindrical functions on $W^T$.  Let
$X_{[0,T]}=\{X_t\colon\ 0\leq t\leq T\}$ for fixed $T>0$. Then, for
$F\in\mathscr{F}C_{0,T}^{\infty}$ with
$F(\gamma)=f(\gamma_{t_1},\ldots, \gamma_{t_n})$, we define the
intrinsic gradient as
$$D_tF(X_{[0,T]}^x)=\sum_{i=1}^{n}\1_{\{t<t_i\}}\,\partrinv t{t_i}\,\nabla_if(X_{t_1}^x,\ldots,X_{t_n}^x), \quad t\in [0,T],$$
where $\nabla_{i}$ denotes the gradient with respect to the $i$-th
component.  The generator $\mathcal{L}$ associated to the Dirichlet
form
$$\mathcal{E}(F,F)=\E\left[\int_0^T|D_tF|^2(X_{[0,T]})\,\vd t\right]$$
is called Ornstein-Uhlenbeck operator.  Let $\gap(\mathcal{L})$ be the
spectral gap of the Ornstein-Uhlenbeck operator $\mathcal{L}$.

In this article, we continue the topic of estimating
gap($\mathcal{L}$) under general lower and upper bounds of the Ricci
curvature.
For the sake of conciseness, let us first introduce some notation: for
constants $K_1$ and $K_2$, define
\begin{align}\label{C-fun}
  C(T,K_1,K_2)=\begin{cases}
    \displaystyle 1+\frac{K_2T}{2}+\frac{K_2^2T^2}{8_{\mathstrut}}, & K_1=0;\\
    \displaystyle (1+\beta)^2-\beta\sqrt{(2+\beta)\left(2+2\beta-\beta\e^{-{K_1T}/{2}}\right)}\e^{-{K_1T}/{4}}, & K_1>0; \\
    \displaystyle
    \frac{1}{2}+\frac{1}{2}\l(1+\beta(1-\e^{-{K_1T}/{2}})\r)^2, &
    K_1<0,
  \end{cases}
\end{align}
where $\beta={K_2}/{K_1}.$

\begin{theorem}\label{th1}
  Let $(M,g)$ be a complete manifold. Assume that
  \begin{align}\label{curvatue-condition}
    k_1\leq \Ric^Z\leq k_2.
  \end{align}
  The following estimate holds:
  \begin{align}\label{Eq:MainEstimate}
    \gap(\mathcal{L})^{-1}\leq C(T, k_1, |k_1|\vee|k_2|)\wedge
    \l[C\l(T,k_1, \frac{k_2-k_1}{2}\r)\times C\l(T,\frac{k_1+k_2}{2}, \frac{|k_1+k_2|}{2}\r)\r].
  \end{align}
\end{theorem}

Let us mention that the first bound in
inequality~\eqref{Eq:MainEstimate}, i.e.,
  $$\gap(\mathscr{L})^{-1}\leq C(T,k_1,|k_1|\vee|k_2|),$$
  is due to Fang and Wu~\cite{FW15}.

\begin{remark}
  In explicit terms we may expand the upper bound as follows:
  \begin{align*}
    C(T&, k_1, |k_1|\vee|k_2|)\\
&=\begin{cases}
 \displaystyle      1+\frac{k_2T}{2}+\frac{k_2^{\,2}T^2}{8}, &  k_1=0;\\
  \displaystyle     \l(\gamma+1\r)^2-\gamma  \sqrt{\l(2+\gamma\r)\l(2\gamma+2-\gamma\e^{-{k_1T}/{2}}\r)}\e^{-{k_1T}/{4}},\qquad\qquad &  k_1>0;\\
 \displaystyle      \frac{1}{2}+\frac{1}{2}\l(1+\gamma-\gamma\e^{-{k_1T}/{2}}\r)^2,&  k_1+k_2\geq 0\ \text{ and }\  k_1<0;\\
 \displaystyle      \frac{1^{\mathstrut}}{2}\l(1+\e^{-k_1T}\r), & k_1+k_2<0,
    \end{cases}
  \end{align*}
  and
  \begin{align*}
    C&\l(T,k_1, \frac{k_2-k_1}{2}\r)\times C\l(T,\frac{k_1+k_2}{2}, \frac{|k_1+k_2|}{2}\r)\\
    &=\begin{cases}
  \displaystyle\l(1+\frac{k_2T}{4}+\frac{k_2^{\,2}T^2}{32}\r)\left(4-(12-3\e^{-\frac{k_2T}{4}})^{{1}/{2}}\,\e^{-\frac{k_2T}{8}}\right),
    & k_1=0;\\
\displaystyle      \frac{1}{4}\left\{\l(\gamma+1\r)^2-\l(\gamma-1\r)
        (\gamma+3)^{{1}/{2}}\big(2\gamma+2-(\gamma-1)\e^{-\frac{k_1T}{2}}\big)^{{1}/{2}}\,\e^{-\frac{k_1T}{4}}\right\}\\
\displaystyle       \qquad\times     \left(4-(12-3\e^{-\frac{k_2T}{4}})^{{1}/{2}}\,\e^{-\frac{(k_1+k_2)T}{8}}\right), & k_1>0;\\
\displaystyle            \frac{1}{2}\left\{1+\frac{1}{4}\l(\gamma+1
        -\l(\gamma-1\r)\e^{-\frac{k_1T}{2}}\r)^2\right\}\\
\displaystyle      \qquad\times \left(4-(12-3\e^{-\frac{k_2T}{4}})^{{1}/{2}}\,\e^{-\frac{(k_1+k_2)T}{8}}\right), 
   &   k_1+k_2\geq 0 \ \text{ and } \ k_1<0;\\
      \frac{1}{4}\left\{1+\frac{1}{4}\l(\gamma+1
        -\l(\gamma-1\r)\e^{-\frac{k_1T}{2}}\r)^2\right\} \l(1+\e^{-\frac{k_1+k_2}{2}T}\r),       &k_1+k_2<0,
    \end{cases}
  \end{align*}
  where $\gamma:=k_2/k_1$.
\end{remark}

By means of Theorem \ref{th1} we are now in position to determine the
asymptotic behavior of $\gap(\mathcal{L})$ as~$T$ tends to~$0$.
\goodbreak

\begin{theorem}\label{th2}
  Assume $k_1\leq \Ric^Z\leq k_2$. Then, as $T\rightarrow 0$, the
  following asymptotics hold:
  \begin{itemize}
  \item [(i)] for $k_1\geq0$,
 $$\gap(\mathscr{L})^{-1}\leq 1+\frac{k_2}{2}T+\frac{1}{8}\l(k_2^2-\frac{(7k_1+k_2)(k_1+k_2)k_2}{6(3k_1+k_2)}\r)T^2+{\rm o}(T^2);$$
\item [(ii)] for $k_1+k_2\geq 0$ and $k_1<0$,
$$\gap(\mathscr{L})^{-1}\leq 1+\frac{k_2}{2}T+\frac{1}{8}\l(k_2^{\,2}+\frac{2k_1^2-k_2^2-5k_1k_2}{6}\r)T^2+{\rm o}(T^2);$$
\item [(iii)] for $k_1+k_2<0$,
$$\gap(\mathscr{L})^{-1}\leq 1-\frac{k_1}{2}T+\frac{1}{8}\l(k_1^2+\frac{3k_1^2+k_2^2}{4}\r)T^2+{\rm o}(T^2).$$
\end{itemize}
\end{theorem}

\begin{remark}\label{rem3}
  Note that as $T\rightarrow 0$, up to the first order, the two upper
  bounds in Theorem~\ref{th1} have the same short-time behaviour,
  however when considered up to second order, our estimates provide
  sharper asymptotics (see the proof of Theorem \ref{th2}).  For
  instance, from \cite[Proposition 3.6]{FW15} we know that if
  $k_1\rightarrow 0$, then
  \begin{equation}\label{est-gap-fw}
    \gap(\mathscr{L})^{-1}\leq 1+\frac{k_2}{2}T+\frac{1}{8}k_2^2T^2+{\rm o}(T^2).
  \end{equation}
  In this case, from Theorem \ref{th2} we deduce that
  $$\gap(\mathscr{L})^{-1}\leq 1+\frac{k_2}{2}T+\frac{5}{48}k_2^2T^2+{\rm o}(T^2)$$
  with a smaller coefficient of $T^2$ when compared to
  estimate~\eqref{est-gap-fw}.
\end{remark}

In Section~\ref{Sect:3} below we shall extend these results to the
path space of an evolving manifold $(M,g_t)$.  Stochastic analysis on
evolving manifolds began with an appropriate notion of Brownian motion
on $(M,g_t)$ (called $g_t$-Brownian motion), see~\cite{ACT}.  Since
then there has been a lot of subsequent work, see for instance,
\cite{KR, Ku, Ku3, Naber, CCM, Ch1, Ch2}.  Here, we deal with
diffusions $X_t$ generated by $L_t=\frac{1}{2}(\Delta_t+Z_t)$ which
are assumed to be non-explosive.  The first author \cite{Ch1}
developed a Malliavin calculus on the path space of $X_t$ by means of
an appropriate derivative formula and an integration by parts
formula. Recently, Naber \cite{Naber} characterized solutions to the
Ricci flow in terms of semigroup gradient estimates. Inspired by this
work, we consider in Section~\ref{Sect:3} one-parameter families of
Ornstein-Uhlenbeck type operators on the path space and adapt a known
method to obtain a family of log-Sobolev inequalities and Poincar\'{e}
inequalities on the path space of the $L_t$-diffusion under a modified
pinched curvature condition.  This curvature contains information
about the time derivative of the metric as well.  In the particular
case of the Ricci flow this modified curvature tensor equals to zero.

The rest of the paper is organized as follows.  In the next section we
establish first a log-Sobolev inequality and a Poincar\'e inequality
on Riemannian path space; these inequalities are the tools to
establish our main results of Section~\ref{Sect:1}.  As already
indicated, Section~\ref{Sect:3} is then devoted to the extension of
the results to evolving manifolds under a geometric flow.

\section{Proofs of main results}\label{Sect:2}

To prove the main results, we introduce a two-parameter family
$\{Q_{r,t}\}_{0\leq r<t}$ of multiplicative functionals as follows:
the $Q_{r,t}$ are a random variable taking values in the linear
automorphisms of $T_{X_r^x}M$ satisfying for fixed $r\geq0$ the
pathwise equation:
\begin{align}\label{e:Q-functional}
  \frac{\vd Q_{r,t}}{\vd t}=-\frac12 Q_{r,t}\,\Ric_{\partr rt}^Z, \quad Q_{r,r}=\id,
\end{align}
where
$\Ric_{\partr rt}^Z=\partrinv rt\circ\Ric_{X_t^x}^Z\circ\partr rt$,
see \cite{Hsu02} and \cite[Theorem 4.1.1]{Wbook2}.  As usual,
$\Ric^Z_x$ operates as linear homomorphism on $T_xM$ via
$\Ric^Z_xv=\Ric^Z(\cdot,v)^\sharp$, $v\in T_xM$.

It is easy to see that if $\Ric^Z\geq K$ for some constant $K$, then
for any $0\leq r\leq t< T$,
$$\|Q_{r,t}\|\leq \e^{- K(t-r)/2},\quad \text{a.s.},$$
where $\|\cdot\|$ denotes the operator norm.  The functionals
$Q_{r,t}$ (or the ``damped parallel transport'' defined as
$\partr rt\circ Q_{r,t}$) are well-known ingredients in the stochastic
representation of the heat flow on one-forms and for Bismut-type
derivative formulas for the diffusion semigroup $\{P_{t}\}_{t\geq 0}$,
see~\cite{Bismut,EL}.

On path space a canonical gradient operator is given in terms of
$Q_{r,t}$.  For any $F\in \mathscr{F}C_{0,T}^{\infty}$ with
$F(\gamma)=f(\gamma_{t_1},\ldots, \gamma_{t_n})$, the damped gradient
$\tilde{D}_tF(X_{[0,T]}^x)$ is defined as
$$\tilde{D}_tF(X^x_{[0,T]})=\sum_{i=1}^{n}\1_{\{t<t_i\}}\,Q_{t,t_i}\,\partrinv t{t_i}\,\nabla_if(X_{t_1}^x,\ldots,X_{t_n}^x), \quad
t\in [0,T].$$
By estimating the damped gradient, a log-Sobolev inequality and a
Poincar\'{e} inequality on path space can be obtained.  Let us first
introduce the following function: for any constants $K_1, K_2$ and
$c$,
\begin{align*}
  \Lambda^c(t,T,K_1,K_2):=\beta(t)+\frac{K_2}{2}\int_0^t\beta(s)\e^{-(\frac{1}{2}K_1+c)(t-s)}\,\vd s,
\end{align*}
where
$\beta(t)=1+\frac{K_2}{2}\int_t^T\e^{-(\frac{1}{2}K_1-c)(s-t)}\vd s$.
Define
\begin{align*}
  S(T,K_1,K_2)=\inf_{c\in \R}\sup_{t\in [0,T]}\Lambda^c(t,T,K_1,K_2).
\end{align*}

\begin{theorem}\label{dd}
  Assume $k_1\leq \Ric^Z\leq k_2$.  Let
  \begin{equation}\label{eq-H}
    H(T,k_1,k_2):=
    S(T, k_1, |k_1|\vee|k_2|) \wedge \l[S\l(T, k_1, \frac{k_2-k_1}{2}\r)S\l( T,\frac{k_2+k_1}{2},\frac{|k_2+k_1|}{2}\r)\r].
  \end{equation}
   Then
  for $F\in\mathscr{F}C_{0,T}^{\infty}$, we have
  \begin{itemize}
  \item [(i)]
    $\E[F^2\log F^2]-\E[F^2]\log \E[F^2]\leq
    2H(T,k_1,k_2)\E\int_0^T|D_tF|^2\,\vd t;$
  \item [(ii)]
    $\E\big[F-\E[F]\big]^2\leq H(T,k_1,k_2)\,\E\int_0^T|D_tF|^2\,\vd
    t.$
  \end{itemize}
\end{theorem}

First, let us introduce some functional inequalities on path space
under pinched curvature condition, which extend the estimates in
\cite{Naber}.  For $F\in \mathscr{F}C_{0,T}^{\infty}$ with
$F(\gamma)=f(\gamma_{t_1},\ldots, \gamma_{t_n})$, we define a modified
gradient as
$$\hat{D}_tF(X_{[0,T]}^x)=\sum_{i=1}^n\1_{\{t\leq t_i\}}\,\e^{-\frac{k_1+k_2}{4}(t_i-t)}\partrinv t{t_i}\, \nabla_if(X_{t_1}^x,\ldots,X_{t_n}^x), \quad
t\in [0,T].$$
In what follows, if there no ambiguity, we write briefly $D_tF$, $\tilde{D}_tF$ and
$\hat{D}_tF$ instead of $D_tF(X_{[0,T]})$, $\tilde{D}_tF(X_{[0,T]})$ and
$\hat{D}_tF(X_{[0,T]})$.

\begin{proposition}\label{lem1}
  Let $(M,g)$ be a complete Riemannian manifold. Let $k_1, k_2$ be two
  real constants such that $k_1\leq k_2$.  The following conditions
  are equivalent:
  \begin{enumerate}[\rm(i)]
  \item $k_1\leq \Ric^Z\leq k_2;$
  \item for any $F\in \mathscr{F}C_{0,T}^{\infty}$,
 $$\big|\nabla_x\E F(X_{[0,T]}^x)\big|\leq \E|\hat{D}_0F|+\frac{k_2-k_1}{4}\int_0^T\e^{-{k_1s}/{2}}\E|\hat{D}_sF|\vd s;$$
\item for any $F\in \mathscr{F}C_{0,T}^{\infty}$ and constant $c$,
    $$\quad\big|\nabla_x\E F(X^x_{[0,T]})\big|^2\leq \l(1+\frac{k_2-k_1}{4}\int_0^T\e^{-(\frac{k_1}{2}-c)s}\vd s\r)\l(\E|\hat{D}_0F|^2+\frac{k_2-k_1}{4}\int_0^T\e^{-(\frac{k_1}{2}+c)s}\E|\hat{D}_sF|^2\vd s\r);$$
  \item for any $F\in \mathscr{F}C_{0,T}^{\infty}$, constant $c$ and
    $t_1<t_2$ in $[0,T]$,
    \begin{align*}
      \quad&\E\Big[\E[F^2(X_{[0,T]})|\mathscr{F}_{t_2}]\log \E[F^2(X_{[0,T]})|\mathscr{F}_{t_2}]\Big]
             -\E\Big[\E[F^2(X_{[0,T]})|\mathscr{F}_{t_1}]\log \E[F^2(X_{[0,T]})|\mathscr{F}_{t_1}]\Big]\\
           &\leq 2\int_{t_1}^{t_2}\l(1+\frac{k_2-k_1}{4}\int_t^T\e^{-(\frac{k_1}{2}-c)(s-t)}\vd s\r)\l(\E|\hat{D}_tF|^2+\frac{k_2-k_1}{4}\int_t^T\e^{-(\frac{k_1}{2}+c)(s-t)}\E|\hat{D}_sF|^2\vd s\r)\vd t;
    \end{align*}
  \item for any $F\in \mathscr{F}C_{0,T}^{\infty}$, constant $c$ and
    $t_1<t_2$ in $[0,T]$,
    \begin{align*}
      \quad&\E\Big[\E[F(X_{[0,T]})|\mathscr{F}_{t_2}]^2\Big]-\E\Big[\E[F(X_{[0,T]})|\mathscr{F}_{t_1}]^2\Big]\\
           &\leq \int_{t_1}^{t_2}\l(1+\frac{k_2-k_1}{4}\int_t^T\e^{-(\frac{k_1}{2}-c)(s-t)}\vd s\r)\l(\E|\hat{D}_tF|^2+\frac{k_2-k_1}{4}\int_t^T\e^{-(\frac{k_1}{2}+c)(s-t)}\E|\hat{D}_sF|^2\vd s\r)\vd t.
    \end{align*}
  \end{enumerate}
\end{proposition}

\begin{proof}
  (a)\, The following inequalities are well known (see \cite{ELL} and
  \cite[Chapter 4]{Wbook2}).  For convenience of the reader we include
  them with precise statements.
  \begin{enumerate}[1)]
  \item for $F\in \mathscr{F}C_{0,T}^{\infty}$, one has
  $$\nabla_x \E[F(X^x_{[0,T]})]=\E[\tilde{D}_0F(X^x_{[0,T]})];$$
\item for $F\in \mathscr{F}C_{0,T}^{\infty}$, one has
  \begin{align*}
    \E&\Big[\E[F^2(X_{[0,T]})|\mathscr{F}_{t_2}]\log \E[F^2(X_{[0,T]})|\mathscr{F}_{t_2}]\Big]
        -\E\Big[\E[F^2(X_{[0,T]})|\mathscr{F}_{t_1}]\log \E[F^2(X_{[0,T]})|\mathscr{F}_{t_1}]\Big]\\
      &\quad\leq 2\E\int_{t_1}^{t_2} |\tilde{D}_tF(X_{[0,T]})|^2\,\vd t;
  \end{align*}
\item for $F\in \mathscr{F}C_{0,T}^{\infty}$, one has
  $$\E\Big[\E[F(X_{[0,T]})|\mathscr{F}_{t_2}]^2\Big]-\E\Big[\E[F(X_{[0,T]})|\mathscr{F}_{t_1}]^2\Big]
  \leq \E\int_{t_1}^{t_2} |\tilde{D}_tF(X_{[0,T]})|^2\,\vd t.$$
\end{enumerate}
Hence it suffices to estimate $|\tilde{D}_tF(X_{[0,T]})|$. For the
sake of brevity, let $k=\frac{k_1+k_2}{2}$ and
$\tilde{k}=\frac{k_2-k_1}{2}$.  It is easy to see that
\begin{align*}
  \tilde{D}_tF
  &=\hat{D}_tF+\sum_{i=1}^{N}\1_{\{t\leq t_i\}}\left(\e^{k(t_i-t)/2}{Q}_{t,t_i}-\id\right)
    \e^{-k(t_i-t)/2}\,\partrinv t{t_i}\,\nabla_i F\\
  &=\hat{D}_tF+\int_t^T\e^{-k(s-t)/2}\frac{\vd \left(\e^{k(s-t)/2}{Q}_{t,s}\right)}{\vd s}\partrinv ts\hat{D}_sF\,\vd s.
\end{align*}
As
\begin{align*}
  \frac{\vd \left(\e^{k(s-t)/2}Q_{t,s}\right)}{\vd s}=-\frac{1}{2}\left(\e^{k(s-t)/2}Q_{t,s}\right)\left(\Ric_{\partr ts}^Z-k\,\id\right),
\end{align*}
we get
\begin{align*}
  \left|\tilde{D}_tF\right|&\leq \left|\hat{D}_tF\right|+\frac{1}{2}\int_t^T\|Q_{t,s}\|\cdot\|(\Ric^Z)^{\sharp}-k\,\id\|\cdot|\hat{D}_sF| \,\vd s\\
                           &\leq  \left|\hat{D}_tF\right|+\frac{1}{2}\tilde{k}\int_t^T\e^{-\frac{1}{2}k_1(s-t)}|\hat{D}_sF| \,\vd s.
\end{align*}
It follows that
\begin{align*}
  \left|\tilde{D}_tF\right|^2&\leq \e^{2ct}\l(\e^{-ct}\big|\hat{D}_tF\big|+\frac{1}{2}\tilde{k}\int_t^T\e^{-(\frac{1}{2}k_1-c)(s-t)}\e^{-cs}|\hat{D}_sF| \,\vd s\r)^2.
\end{align*}
Thus, by Cauchy's inequality, we obtain
\begin{align*}
  \left|\tilde{D}_tF\right|^2&\leq \e^{2ct}\l(1+\frac{1}{2}\int_t^T\tilde{k}\e^{-(\frac{1}{2}k_1-c)(s-t)}\,\vd s\r)\l(\e^{-2ct}\left|\hat{D}_tF\right|^2+\int_t^T\frac{1}{2}\tilde{k}\e^{-(\frac{1}{2}k_1-c)(s-t)-2cs}|\hat{D}_sF|^2\,\vd s\r)\\
                             &=\l(1+\frac{1}{2}\int_t^T\tilde{k}\e^{-(\frac{1}{2}k_1-c)(s-t)}\,\vd s\r)\l(\left|\hat{D}_tF\right|^2+\int_t^T\frac{1}{2}\tilde{k}\e^{-(\frac{1}{2}k_1+c)(s-t)}|\hat{D}_sF|^2\,\vd s\r).
\end{align*}
This allows to complete the proof of (i) implies (ii)--(v).

(b)\, Conversely, to prove $\text{(ii)--(v)}\Rightarrow\text{(i)}$, by
a similar argument as in \cite{Naber, WW}, it suffices to prove that
(iii) implies (i). Following \cite{Naber}, we first take
$F(X_{[0,T]}^x)=f(X_t^x)$ as test functional. In this case, (iii)
reduces to
\begin{align}\label{equi-1}
  & |\nabla P_tf|^2(x) \leq \l[\l(1+\frac{\tilde{k}}{2}\int_0^t\e^{-(\frac{k_1}{2}-c)r}\vd r\r)\l(1+\frac{\tilde{k}}{2}\int_0^t\e^{-(\frac{k_1}{2}+c-k)r}\vd r\r)\e^{-kt}\r] P_t|\nabla f|^2(x).
\end{align}
By means of the formula from \cite[Theorem 2.2.4]{Wbook2}:
\begin{align*}
  \Ric^Z(\nabla f,\, \nabla f)(x)
  &=\lim_{t\rightarrow 0}\frac{P_{t}|\nabla f|^2(x)-|\nabla P_{t}f|^2(x)}{t},\quad f\in C_0^{\infty}(M),
\end{align*}
we obtain the inequality $\Ric^Z\geq k_1$.  Taking however
$F(X_{[0,T]}^x)=f(x)-\frac{1}{2}f(X_t^x)$ as test functional, then
(iii) reduces to the inequality:
\begin{align*}
  &\l|\nabla f(x)-\frac{1}{2}\nabla P_tf(x)\r|^2
    \leq\l(1+\frac{k_2-k_1}{4}\int_0^t\e^{-(\frac{k_1}{2}-c)s}\vd s\r)\\
  &\qquad\qquad\times\l(\E\big|\nabla f(x)-\frac{1}{2}\e^{-\frac{k}{2}t}\partrinv0t\,\nabla f(X_t^x)\big|^2+\frac{k_2-k_1}{16}\l(\int_0^t\e^{-(c-\frac{k_2}{2})s}\vd s\r)\e^{-kt}\,P_t|\nabla f|^2(x)\r).
\end{align*}
Expanding the last inequality, we arrive at
\begin{align}\label{equi-2}
  |\nabla P_tf(x)|^2&-\l[\l(1+\frac{\tilde{k}}{2}\int_0^t\e^{-(\frac{k_1}{2}-c)r}\vd r\r)\l(1+\frac{\tilde{k}}{2}\int_0^t\e^{-(c-\frac{k_2}{2})r}\vd r\r)\e^{-kt}\r] P_t|\nabla f|^2(x)\notag\\
                    &\leq (k_2-k_1)\,\int_0^t\e^{-(\frac{k_1}{2}-c)s}\vd s\,|\nabla f(x)|^2 + 4\l<\nabla f(x), \nabla P_tf(x)\r>\notag\\
                    &\quad -4\l(1+\frac{k_2-k_1}{4}\int_0^t\e^{-(\frac{k_1}{2}-c)s}\vd s\r)\e^{-\frac{k}{2}t}
                      \big<\nabla f(x), \E\partrinv 0t\nabla f(X_t^x)\big>.
\end{align}
Then by \cite[Lemma 2.5]{Ch-Th:2016a} it is straightforward to derive
the upper bound $\Ric^Z\leq k_2$.
\end{proof}

\begin{remark}\label{rem4}
  In our forthcoming paper \cite{Ch-Th:2016a} we use a direct method
  which does not need to use the advanced theory on path space, to
  prove the result that the pinched curvature condition is equivalent
  to the coupled conditions \eqref{equi-1} and \eqref{equi-2} when
  $c=(k_1+k_2)/{4}$.
\end{remark}

\begin{proof}[Proof of Theorem \ref{dd}]
  The following inequalities are well known (see \cite{ELL} and
  \cite[Chapter 4 ]{Wbook2}).  For convenience of the reader we
  include them here, as we have done in the proof of
  Proposition~\ref{lem1}.
  \begin{enumerate}[1)]
  \item for $F\in \mathscr{F}C_{0,T}^{\infty}$, one has
  $$\E[F^2\log F^2]-\E[F^2]\log \E[F^2]\leq 2\E\int_0^T |\tilde{D}_tF(X_{[0,T]})|^2\,\vd t;$$
\item for $F\in \mathscr{F}C_{0,T}^{\infty}$, one has
  $$\E\big[F-\E[F]\big]^2\leq \E\int_0^T |\tilde{D}_tF(X_{[0,T]})|^2\,\vd t.$$
\end{enumerate}
Hence, it suffices to estimate $\E\int_0^T |\tilde{D}_tF|^2\,\vd t$
where $\tilde{D}_tF=\tilde{D}_tF(X_{[0,T]})$. By \cite{Naber}, we know
that
\begin{align*}
  |\tilde{D}_tF|\leq \big|{D}_tF\big|+\frac{|k_1|\vee|k_2|}{2}\int_t^T\e^{-\frac{1}{2}k_1(s-t)}|{D}_sF| \,\vd s.
\end{align*}
It follows that for any constant $c$, we have
\begin{align*}
  \left|\tilde{D}_tF\right|^2&\leq \e^{2ct}\l(\e^{-ct}\big|{D}_tF\big|+\frac{|k_1|\vee|k_2|}{2}\int_t^T\e^{-(\frac{1}{2}k_1-c)(s-t)}\e^{-cs}|{D}_sF| \,\vd s\r)^2.
\end{align*}
Thus, by Cauchy's inequality, we obtain
\begin{align}\label{eq:esti-1}
  \left|\tilde{D}_tF\right|^2\leq \l(1+\frac{|k_1|\vee|k_2|}{2}\int_t^T\e^{-(\frac{1}{2}k_1-c)(s-t)}\,\vd s\r)\l(\left|{D}_tF\right|^2+\frac{|k_1|\vee|k_2|}{2}\int_t^T\e^{-(\frac{1}{2}k_1+c)(s-t)}|{D}_sF|^2\,\vd s\r).
\end{align}
Let
$$\alpha_1(t)=1+\frac{|k_1|\vee|k_2|}{2}\int_t^T\e^{-(\frac{1}{2}k_1-c)(s-t)}\,\vd
s.$$
Then, integrating both sides of Eq.~\eqref{eq:esti-1} from $0$ to $T$
yields
\begin{align*}
  \int_0^T\left|\tilde{D}_tF\right|^2\,\vd t&\leq \int_0^T\alpha_1(t)\left(\left|{D}_tF\right|^2+\int_t^T\frac{|k_1|\vee|k_2|}{2}\e^{-(\frac{1}{2}k_1+c)(s-t)}|{D}_sF|^2\,\vd s\r) \vd t\\
                                            &=\int_0^T\l(\alpha_1(t)+\frac{|k_1|\vee |k_2|}{2}\int_0^t\alpha_1(s)\e^{-(\frac{1}{2}k_1+c)(t-s)}\,\vd s\r)|D_tF|^2\, \vd t\\
                                            &=\int_0^T\Lambda^c(t,T,k_1,|k_1|\vee|k_2|)|D_tF|^2\, \vd t\\
                                            &\leq S(T,k_1,|k_1|\vee|k_2|)\int_0^T|D_tF|^2\, \vd t.
\end{align*}

We are now going to prove
\begin{align*}
  \E\int_0^T |\tilde{D}_tF(X_{[0,T]})|^2\,\vd t\leq S\l(T,k_1, \frac{k_2-k_1}{2}\r)S\l(T,\frac{k_1+k_2}{2}, \frac{|k_2+k_1|}{2}\r)\int_0^T\E|D_tF(X_{[0,T]})|^2\,\vd t.
\end{align*}
Our first step is to show that
\begin{align*}
  \E\int_0^T |\tilde{D}_tF(X_{[0,T]})|^2\,\vd t\leq S\l(T,k_1,\frac{k_2-k_1}{2}\r) \int_0^T \E|\hat{D}_tF(X_{[0,T]})|^2\,\vd t.
\end{align*}
Recall the notations introduced above
$$\hat{D}_tF(X_{[0,T]}):=\sum_{i=1}^n\1_{\{t\leq t_i\}}\,\e^{-\frac{k_1+k_2}{4}(t_i-t)}\,\partrinv t{t_i}\,\nabla_i f(X_{t_1},\ldots,X_{t_n})$$
and $k:=\frac{k_1+k_2}{2},\, \tilde{k}:=\frac{k_2-k_1}{2}$.  By
Proposition \ref{lem1}, for any constant $c$, we have
\begin{align*}
  \left|\tilde{D}_tF\right|^2\leq \l(1+\frac{1}{2}\int_t^T\tilde{k}\e^{-(\frac{1}{2}k_1-c)(s-t)}\,\vd s\r)\l(\left|\hat{D}_tF\right|^2+\int_t^T\frac{1}{2}\tilde{k}\e^{-(\frac{1}{2}k_1+c)(s-t)}|\hat{D}_sF|^2\,\vd s\r).
\end{align*}
Integrating both sides from $0$ to $T$ yields
\begin{align*}
  \int_0^T\left|\tilde{D}_tF\right|^2\,\vd t&\leq \int_0^T\left(1+\frac{1}{2}\int_t^T\tilde{k}\e^{-(\frac{1}{2}k_1-c)(s-t)}\,\vd s\r)\left(\left|\hat{D}_tF\right|^2+\int_t^T\frac{1}{2}\tilde{k}\e^{-(\frac{1}{2}k_1+c)(s-t)}|\hat{D}_sF|^2\,\vd s\r) \vd t.
\end{align*}
Let
$\alpha_2(t)=1+\frac{1}{2}\int_t^T\tilde{k}\e^{-(\frac{1}{2}k_1-c)(s-t)}\vd
s$. Then
\begin{align*}
  \int_0^T\left|\tilde{D}_tF\right|^2\vd t\leq& \int_0^T \alpha_2(t)\left(\left|\hat{D}_tF\right|^2+\int_t^T\frac{1}{2}\tilde{k}\e^{-(\frac{1}{2}k_1+c)(s-t)}|\hat{D}_sF|^2\,\vd s\r) \,\vd t\\
  =&\int_0^T\alpha_2(t)\left|\hat{D}_tF\right|^2\,\vd t+\int_0^T\alpha_2(t)\int_t^T \frac{1}{2}\tilde{k}\e^{-(\frac{1}{2}k_1+c)(s-t)}|\hat{D}_sF|^2\,\vd s \,\vd t\\
  =& \int_0^T\left(\alpha_2(t)+\frac{1}{2}\tilde{k}\int_0^t\alpha_2(s) \e^{-(\frac{1}{2}k_1+c)(t-s)}\,\vd s\r)\left|\hat{D}_tF\right|^2\,\vd t\\
  =&\int_0^T\Lambda^c(t,T,k_1, \tilde{k})\left|\hat{D}_tF\right|^2\,\vd t.
\end{align*}
Therefore, we have
$$\int_0^T\left|\tilde{D}_tF\right|^2\vd t\leq \inf_{c\in \R}\sup_{t\in [0,T]}\Lambda^c(t,T, k_1, \tilde{k}) \int_0^T\left|\hat{D}_tF\right|^2 \vd t.$$
Our second step is to prove
\begin{align*}
  \int_0^T\E\left|\hat{D}_tF\right|^2\,\vd t\leq S(T, k, |k|)\int_0^T\E|{D}_tF|^2\,\vd t.
\end{align*}
To this end, we first observe that
\begin{align*}
  \left|\hat{D}_tF\right|&=\l|\sum_{i=1}^{N}\1_{\{t\leq t_i\}}\e^{-\frac{1}{2}k(t_i-t)}\,\partrinv t{t_i}\,\nabla_i F\r|
                         \leq |D_t F|+\frac{1}{2}|k|\int_t^T\e^{-\frac{1}{2}k(s-t)}|D_sF|\,\vd s.
\end{align*}
Let
$\alpha_3(t)=1+\frac{1}{2}|k|\int_t^T\e^{-(\frac{1}{2}k-c)(s-t)}\,\vd
s$ for some constant $c$. We have
\begin{align*}
  \int_0^T\left|\hat{D}_tF\right|^2\,\vd t&\leq \int_0^T\l(|D_t F|+\frac{1}{2}|k|\int_t^T\e^{-\frac{1}{2}k(s-t)}|D_sF|\,\vd s\r)^2 \,\vd t\\
                                          &\leq \int_0^T\l(\alpha_3(t)+\frac{1}{2}|k|\int_0^t\alpha_3(s)\e^{-(\frac{1}{2}k+c)(t-s)}\,\vd s\r)|D_tF|^2\,\vd s.
\end{align*}
It is easy to see that
$$\Lambda^c(t,T, k, |k|)=\alpha_3(t)+\frac{1}{2}|k|\int_0^t\alpha_3(s)\e^{-(\frac{1}{2}k+c)(t-s)}\,\vd s.$$
Hence, we arrive at
\begin{equation*}
  \int_0^T\E\left|\tilde{D}_tF\right|^2\,\vd t \leq S(T, k_1,\tilde{k})S(T, k, |k|)\int_0^T\E |{D}_tF|^2\,\vd t,\qedhere
\end{equation*}
which completes the proof of Theorem \ref{dd}.
\end{proof}

In the proof of Theorem~\ref{th1} the function $\Lambda:=\Lambda^0$
will play an important role. More precisely, for constants $K_1$ and
$K_2$, we have
\begin{align*}\Lambda&(t,T, K_1, K_2)\\
                     &\quad=\begin{cases}
                       \displaystyle
                       (1+\beta)^2-(\beta+\beta^2)\e^{-{K_1t}/{2}}-\frac{2\beta+\beta^2}{2}\e^{-{K_1(T-t)}/{2}}
                       +\frac{\beta^2}{2}\e^{-{K_1(T+t)}/{2}}, & \hbox{if}\ K_1\neq 0, \\
                       \displaystyle 1+\frac{K_2T}{2}+\frac{K_2^2}{8}(2Tt-t^2), &
                       \hbox{if}\ K_1=0
                     \end{cases}
\end{align*}
where $\beta={K_2}/{K_1}$.  We choose here the value $c=0$, since it
seems to give the best asymptotics as $T\to 0$.

\begin{proposition}\label{prop1}
  Let $K_1$ and $K_2$ be two constants such that $K_2\geq 0$. Then
$$
C(T,K_1,K_2)=\sup_{t\in [0,T]}\Lambda(t,T,K_1,K_2),
$$
where $C(T,K_1,K_2)$ is defined as in \eqref{C-fun}.
\end{proposition}

\begin{proof}
  For the case $K_1+K_2\geq 0$, the reader is referred to
  \cite[Proposition 3.3]{FW15}. It suffices to deal with the remaining
  case $K_1+K_2<0$.  The idea is similar to the proof of
  \cite[Proposition 3.3]{FW15}.

  When $K_1+K_2<0$ and $K_2\geq 0$, we must have $K_1<0$. Taking
  derivative of $\Lambda$ with respect to $t$, we obtain
  \begin{align*}
    \Lambda'(t,T,K_1,K_2)=\frac{K_1}{4}\e^{-{K_1t}/{2}}\l[2(\beta+\beta^2)-\beta^2\e^{-{K_1T}/{2}}-(2\beta+\beta^2)\e^{-{K_1T}/{2}}\e^{K_1t}\r],
  \end{align*}
  where $\beta=K_2/K_1$.  From this it is easy to see that there
  exists at most one point $t$ such that $$\Lambda'(t,T,K_1,K_2)=0.$$
  In addition, for the boundary values $t=0,T$, we have
  \begin{align*}
    \Lambda'(0,T,K_1,K_2)&=\frac{1}{2}\beta(K_1+K_2)\left(1-\e^{-{K_1T}/{2}}\right)<0;\\
    \Lambda'(T,T,K_1,K_2)&=-\frac{K_2}{2}\left(1-\e^{-{K_1T}/{2}}\right)-\frac{K_2^2}{4K_1}\left(1-\e^{-{K_1T}/{2}}\right)^2>0.
  \end{align*}
  Thus, we obtain that the maximal value of $\Lambda$ over the
  interval $[0,T]$ is reached either at $t=0$ or at $t=T$.  Moreover,
  by inspection it is easy to see that
  $\Lambda(0,T,K_1,K_2)\leq
  \frac{1}{2}+\frac{1}{2}\Lambda^2(0,T,K_1,K_2)=\Lambda(T,T,K_1,K_2)$.
  All this taken together, we may conclude that
  \begin{equation*}
    \sup_{t\in [0,T]}\Lambda(t,T,K_1,K_2)=\Lambda(T,T,K_1,K_2).\qedhere
  \end{equation*}
\end{proof}

\begin{proof}[Proof of Theorem \ref{th1}]
  From Theorem \ref{dd} we conclude that
  \begin{align}\label{general-esti}
\gap(\mathcal{L})^{-1}\leq {H(T,k_1,k_2)}.
  \end{align}
Moreover, it is easy to be observed that
$$S\l(T, K_1,K_2\r)\leq \sup_{t\in [0,T]}\Lambda\l(t,T, K_1,K_2\r)=C\l(T, K_1,K_2\r),$$
which allows to complete the proof of Theorem \ref{th1}.
\end{proof}

\begin{proof}[Proof of Theorem \ref{th2}]
  We check the short-time behavior of $C(T,K_1, K_2)$ for $K_2\geq 0$
  first. If $K_1>0$, then
  \begin{align*}
    C(T,K_1,K_2)&=(1+\beta)^2-\beta\sqrt{(2+\beta)\left(2+2\beta-\beta\e^{-{K_1T}/{2}}\right)}\e^{-{K_1T}/{4}}\\
                &=(1+\beta)^2-\beta(2+\beta)\e^{-{K_1T}/{4}}\sqrt{1+\frac{\beta}{2+\beta}\left(1-\e^{-{K_1T}/{2}}\right)}.
  \end{align*}
  Note that
  \begin{align*}
    \sqrt{1+\frac{\beta}{2+\beta}(1-\e^{-{K_1T}/{2}})}&=1+\frac{\beta}{2(2+\beta)}(1-\e^{-{K_1T}/{2}})
                                                        -\frac{1}{8}\frac{\beta^2}{(2+\beta)^2}(1-\e^{-{K_1T}/{2}})^2+{\rm o}(T^2)\\
                                                      &=1+\frac{\beta}{2(2+\beta)}\l(\frac{1}{2}K_1T-\frac{1}{8}(K_1T)^2\r)-
                                                        \frac{\beta^2}{32(2+\beta)^2}(K_1T)^2+{\rm o}(T^2)\\
                                                      &=1+\frac{\beta}{4(2+\beta)} K_1T-
                                                        \frac{\beta(4+3\beta)}{32(2+\beta)^2}(K_1T)^2+{\rm o}(T^2).
  \end{align*}
  Thus,
  \begin{align*}
    C(T,K_1,K_2)&=(1+\beta)^2-\beta(2+\beta)\l(1-\frac{1}{4}K_1T+\frac{1}{32}(K_1T)^2+{\rm o}(T^2)\r)\\
                &\hspace{3cm} \times\l(1+\frac{\beta}{4(2+\beta)} K_1T-
                  \frac{\beta(4+3\beta)}{32(2+\beta)^2}(K_1T)^2+{\rm o}(T^2)\r)\\
                &=1+\frac{K_2T}{2}+\l(1-\frac{(K_1+K_2)K_1}{(2K_1+K_2)K_2}\r)\frac{K_2^2T^2}{8}+{\rm o}(T^2).
  \end{align*}
  If $K_1<0$, then
  \begin{align*}
    C(T,K_1,K_2)&=\frac{1}{2}+\frac{1}{2}\l(1+\beta(1-\e^{-{K_1T}/{2}})\r)^2\\
                &=\frac{1}{2}+\frac{1}{2}\l(1+\beta\frac{K_1 T}{2}-\beta\frac{(K_1T)^2}{8}+{\rm o}(T^2)\r)^2\\
                &=1+\frac{1}{2}K_2T+\l(1-\frac{K_1}{K_2}\r)\frac{(K_2T)^2}{8}+{\rm o}(T^2).
  \end{align*}
  Hence, for $C(T,k_1, |k_1|\vee |k_2|)$, we obtain
  \begin{align*}
    C(T,k_1, |k_1|\vee |k_2|)=\begin{cases}
      \displaystyle 1+\frac{k_2}{2}T+\frac{k_2^2}{8_{\mathstrut}}T^2-\frac{k_1k_2(k_1+k_2)}{8(2k_1+k_2)}T^2+{\rm o}(T^2),& k_1\geq 0,\\
      \displaystyle1+\frac{k_2}{2}T+\frac{k_2^2}{8_{\mathstrut}}T^2-\frac{k_1k_2}{8}T^2+{\rm o}(T^2),& k_1+k_2\geq 0 \ \text{ and }\ k_1<0,\\
      \displaystyle1-\frac{k_1}{2}T+\frac{k_1^2}{8}T^2+\frac{(k_1T)^2}{8}+{\rm o}(T^2),& k_1+k_2<0.
    \end{cases}
  \end{align*}
  We now turn to estimate
  $\displaystyle
  C\big(T,k_1,\frac{k_2-k_1}{2}\big)\,C\big(T,\frac{k_1+k_2}{2},\frac{|k_2+k_1|}{2}\big)$.
  \begin{enumerate}[(i)]
  \item When $k_1+k_2<0$, we have
    \begin{align*}
      &C\l(T,k_1,\frac{k_2-k_1}{2}\r)C\l(T,\frac{k_2+k_1}{2},-\frac{k_2+k_1}{2}\r)\\
      &=1+\frac{k_2}{2}T+\frac{k_2^2}{8}T^2+\frac{3k_1^2+k_2^2}{32}T^2+{\rm o}(T^2);
    \end{align*}
  \item when $k_1+k_2\geq 0$ and $k_1\leq 0$,
    \begin{align*}
      &C\l(T,k_1,\frac{k_2-k_1}{2}\r)C\l(T,\frac{k_2+k_1}{2},\frac{k_2+k_1}{2}\r)\\
      &=1+\frac{k_2}{2}T+\frac{k_2^2}{8}T^2+\frac{2k_1^2-k_2^2-5k_1k_2}{48}T^2+{\rm o}(T^2);
    \end{align*}
  \item when $k_1> 0$,
    \begin{align*}
      &C\l(T,k_1,\frac{k_2-k_1}{2}\r)C\l(T,\frac{k_2+k_1}{2},\frac{k_2+k_1}{2}\r)\\
      &=1-\frac{k_1}{2}T+\frac{k_1^2}{8}T^2-\frac{(7k_1k_2+k_2^2)(k_1+k_2)}{48(3k_1+k_2)}T^2+{\rm o}(T^2).
    \end{align*}
  \end{enumerate}
  Summarizing the estimates above, we conclude that as
  $T\rightarrow 0$,
 $$C(T,k_1, |k_1|\vee |k_2|)\ \text{ and }\ C\l(T,k_1,\frac{k_2-k_1}{2}\r)C\l(T,\frac{k_1+k_2}{2},\frac{|k_2+k_1|}{2}\r)$$ have the same first order term, i.e.~coefficient of $T$, and we only need to compare the coefficients of~$T^2$.
 \begin{enumerate}[(i)]
 \item If $k_1\geq 0$, then
$$-\frac{(7k_1k_2+k_2^2)(k_1+k_2)}{48(3k_1+k_2)}+\frac{k_1k_2(k_1+k_2)}{8(2k_1+k_2)}=-\frac{(k_2^2-k_1^2)(4k_1+k_2)k_2}{48(3k_1+k_2)(2k_1+k_2)}\leq 0.$$
\item If $k_1+k_2\geq 0$ and $k_1<0$, then
$$\frac{2k_1^2-k_2^2-5k_1k_2}{48}+\frac{k_1k_2}{8}=\frac{(2k_1-k_2)(k_1+k_2)}{48}\leq 0.$$
\item If $k_1+k_2<0$, then
$$\frac{3k_1^2+k_2^2}{32}-\frac{k_1^2}{8}=\frac{k_2^2-k_1^2}{32}<0.$$
\end{enumerate}
From this we conclude that
$$C\left(T,k_1,\frac{k_2-k_1}{2}\right)\,C\left(T,\frac{k_1+k_2}{2},\frac{|k_2+k_1|}{2}\right)$$
has a smaller coefficient in $T^2$.  The proof is then completed by
using Theorem \ref{th1}.
\end{proof}

\section{Extension to the  path space of evolving manifolds}\label{Sect:3}
In this section, our base space is a differentiable manifold carrying
a geometric flow of complete Riemannian metrics, more precisely, a
$d$-dimensional differential manifold $M$ equipped with a family of
complete Riemannian metrics $(g_t)_{t\in [0,T_c)}$ for some
$T_c\in (0,\infty]$, which is $C^1$ in $t$.

Let $\nabla^t$ and $\Delta_t$ be the Levi-Civita connection and the
Laplace-Beltrami operator associated with the metric $g_t$,
respectively. Let $(Z_t)_{t\in [0,T_c)}$ be a $C^{1,\infty}$-family of
vector fields.  Consider the diffusion process $X_t^{x}$ generated by
$L_t=\frac{1}{2}(\Delta_t+Z_t)$ (called $L_t$-diffusion process)
starting from $x$ at time $0$, which is assumed to be non-explosive
before $T_c$ (see \cite{KR} for sufficient conditions).

It is well-known (e.g.~\cite{ACT,Elworthy:1988}) that $X_t^{x}$ solves
the equation
$$\vd X_t^{x}= u_t^x\circ \vd B_t+\frac{1}{2}Z_t(X_t^{x})\vd t,\quad X_0^{x}=x=\pi (u^x_0),$$
where $B_t$ is an $\R^d$-valued Brownian motion on a filtered
probability space $(\Omega,\{\mathscr{F}_t\}_{t\geq 0}, \P)$
satisfying the usual conditions.  Here $u^x_t$ is a horizontal process
above $X^x_t$ taking values in the frame bundle over~$M$, constructed
in such a way that the parallel transports
$$\partr st:=u^x_t\circ (u^x_s)^{-1}\colon(T_{X^x_s}M,g_s)\to(T_{X^x_t}M,g_t),\quad s\leq t,$$
along the paths of $X$ are isometries, see~\cite{ACT} for the
construction, as well as Section 3 in \cite{Ch-Th:2016a} for some
details.

By It\^{o}'s formula, for any $f\in C_0^2(M)$ and $t\in [0,T_c)$, the
process
$$f(X_t^{x})-f(x)-\int_0^tL_rf(X^{x}_r)\,\vd r=\int_0^t\big< \partrinv0r\nabla^rf(X^{x}_r), u_0^x\,\vd B_r\big >_0$$
is a martingale up to $T_c$, where $\l<\cdot,\cdot\r>_0$ is the inner
product on $T_xM$ given by the initial metric $g_0$. In other words,
$X_t^{x}$ is a diffusion generated by $L_t$.

For the sake of brevity, we introduce the following notation: for
$X, Y\in TM$ such that $\pi(X)=\pi(Y)$ let
\begin{align*}
  &\mathcal{R}_t^Z(X,Y):=\Ric_t(X,Y)-\l<\nabla^t_XZ_t, Y\r>_t-\partial_tg_t(X,Y)
\end{align*}
where $\Ric_t$ is the Ricci curvature tensor with respect to the
metric $g_t$ and $\l<\cdot,\cdot\r>_t=g_t(\cdot,\cdot)$. In what
follows, given functions $\phi,\psi$ on $[0,T_c)\times M$, we write
$\psi\leq \mathcal{R}_t^Z\leq \phi$ if
$$\psi|X|^2_t\leq \mathcal{R}_t^Z(X,X)\leq \phi |X|^2_t$$ holds for
all $X\in TM$, where $|X|_t:=\sqrt{g_t(X,X)}$.

Similarly to Eq.~\eqref{e:Q-functional} we define a two-parameter
family of multiplicative functionals $\{Q_{r,t}\}_{r\leq t}$ as
solution to the following equation: for $0\leq r\leq t<T_c$ let
\begin{align}\label{e:Q-functional-1}
  \frac{\vd Q_{r,t}}{\vd t}=-\frac{1}{2}Q_{r,t}\,{\mathcal R}^Z_{\partr rt}\,,\quad Q_{r,r}=\id,
\end{align}
where by definition
$${\mathcal R}^Z_{\partr rt}:={\partrinv rt}\circ{\mathcal R}^Z_t(X_t^x)\circ{\partr rt}.$$

Let $W(M)$ be the path space of $M$. Fixing $T\in (0,T_c)$ we have the
space of smooth cylindrical functions on $W(M)$ defined as
$$\mathscr{F}C_{0,T}^{\infty}:=\left\{F\colon F=f(X_{t_1},\ldots, X_{t_n}),\
  0<t_1<\ldots<t_n\leq T,\ f\in C_0^{\infty}(M^n)\right\}.$$
For $F\in \mathscr{F}C_{0,T}^{\infty}$ we consider again different
types of gradients:
\begin{enumerate}[(i)]
\item \textit{intrinsic gradient}:
$$D_tF(X_{[0,T]})=\sum_{i=1}^n\1_{\{t\leq t_i\}}\,\partrinv t{t_i}\,\nabla^{t_i}_i f(X_{t_1},\ldots, X_{t_n}), \quad t\in [0,T];$$
\item \textit{damped gradient}:
$$\tilde{D}_tF(X_{[0,T]})=\sum_{i=1}^{n}\1_{\{t\leq t_i\}}\,Q_{t,t_i}\,\partrinv t{t_i}\,\nabla^{t_i}_i f(X_{t_1},\ldots, X_{t_n}),\quad t\in [0,T];$$
\item \textit{modified gradient}:
$$\hat{D}_tF(X_{[0,T]})=\sum_{i=1}^n\1_{\{t\leq
  t_i\}}\,\e^{-\frac{1}{4}{\textstyle\int_t^{t_i}}(k_1+k_2)(r)\vd
  r}\,\partrinv t{t_i}\,\nabla^{t_i}_i f(X_{t_1},\ldots,X_{t_n}).$$
\end{enumerate}
We again write briefly $D_tF$, $\tilde{D}_tF$ and $\hat{D}_tF$ instead
of $D_tF(X_{[0,T]})$, $\tilde{D}_tF(X_{[0,T]})$ and
$\hat{D}_tF(X_{[0,T]})$ if there no ambiguity.  In terms of the
intrinsic gradient $D_0$, we consider the one-parameter family of
Ornstein-Uhlenbeck type operators given as
$$\mathcal{L}=-D_0^{\,*}D_0.$$
Our aim is to give an estimate for the spectral gap of $\mathcal{L}$,
denoted by $\gap(\mathcal{L})$.  To this end, we use the Poincar\'{e}
inequality and log-Sobolev inequality of the next theorem. For the
precise statement some notation is required.  Given three functions
$K_1$, $K_2$ and $c$ in $C([0,T])$, we define
$$\tilde{\Lambda}^c(t,T,K_1,K_2)
=\alpha(t)+\frac{1}{2}K_2(t)\int_0^t\alpha(s)\e^{-\frac{1}{2}{\textstyle\int_s^t}(K_1+2c)(r)\vd
  r}\vd s$$ where
$$\alpha(t)=1+\frac{1}{2}\int_t^TK_2(s)\e^{-\frac{1}{2}{\textstyle\int_t^s}(K_1-2c)(r)\vd
  r}\vd s.$$ Furthermore let
\begin{align*}
  &\tilde{S}(T,K_1,K_2)=\inf_{c\in C([0,T])}\sup_{t\in
    [0,T]}\tilde{\Lambda}^c(t,T,K_1,K_2).
\end{align*}
Note that if $K_1, K_2, c$ are constants then
$$\tilde{\Lambda}^c(t,T,K_1,K_2)={\Lambda}^c(t,T,K_1,K_2).$$

Analogously to Theorem \ref{dd} recall the following two inequalities.

\begin{theorem}\label{Fol-dd}
  Assume that there exist continuous functions $k_1, k_2$ such that
  for every vector field~$X$,
  \begin{equation}\label{curvature-cond}
    k_1(t)\,|X|_t^2\leq \mathscr{R}_t^Z(X,X)\leq k_2(t)\,|X|_t^2,\quad t\in [0,T].
  \end{equation}
  Then,
  \begin{enumerate}[\rm(i)]
  \item for every cylindrical function
    $F\in \mathscr{F}C_{0,T}^{\infty}$,
$$\E[F^2\log F^2]-\E [F^2]\log \E [F^2]\leq 2 \tilde{H}(T,k_1, k_2)\int_0^T\E|D_sF|^2_{s}\,\vd s,$$
where
$$\tilde{H}(T,k_1,k_2)=\tilde{S}(T,k_1,|k_1|\vee|k_2|)\wedge
\l[\tilde{S}\l(T,k_1,\frac{k_2-k_1}{2}\r)\tilde{S}\l(T,\frac{k_1+k_2}{2},\frac{|k_1+k_2|}{2}\r)\r];$$

\item for every cylindrical function
  $F\in \mathscr{F}C_{0,T}^{\infty}$,
$$\E\big[F-\E[F]\big]^2\leq \tilde{H}(T,k_1, k_2)\int_0^T\E|D_sF|^2_{s}\,\vd s.$$
\end{enumerate}
\end{theorem}

Similarly to Section 2, we need the characterizations of modified
pinched curvature condition on path space to prove Theorem
\ref{Fol-dd}. In the following, we will use the notation:
$$\E^{(x,t)}[\,\bolddot\,]:=\E[\,\bolddot\,|\mathscr{F}_t,\, X_t=x].$$

\begin{proposition}\label{lem2}
  Let $(M,g_t)_{t\in [0,T_c)}$ be a smooth manifold carrying a family
  of complete metrics $g_t$. Let $k_1, k_2$ be two continuous functions in $C([0,T_c))$ such
  that $k_1\leq k_2$. For any $T\in (0,T_c)$, the following conditions
  are equivalent:
  \begin{enumerate}[\rm (i)]
  \item for any $t\in [0,T]$,
   $$k_1(t)\leq \mathcal{R}_t^Z\leq k_2(t);$$
 \item for any $F\in \mathscr{F}C_{0,T}^{\infty}$,
 $$\big|\nabla_{x}^t\,\E^{(x,t)} (F(X_{[0,T]}))\big|_t\leq \E^{(x,t)}(|\hat{D}_tF|_t)
 +\frac{1}{2}\int_t^T\tilde{k}(s)\e^{-\frac{1}{2}{\textstyle\int_t^s}
   k_1(r)\vd r}\E^{(x,t)}(|\hat{D}_sF|_s)\vd s$$
 where $\tilde{k}=(k_2-k_1)/2;$
\item for any $F\in \mathscr{F}C_{0,T}^{\infty}$ and any continuous
  function $c$ on $[0,T]$,
  \begin{align*}
    \big|\nabla_{x}^t\,\E^{(x,t)} (F(X_{[0,T]}))\big|_t^2
    &\leq \l(1+\frac{1}{2}\int_t^T\tilde{k}(s)\e^{-{\textstyle\int_t^s} (\frac{1}{2} k_1(r)-c(r))\vd r}\vd s\r)\\
    &\qquad\times\l(\E^{(x,t)}(|\hat{D}_tF|_t^2)+\frac{1}{2}\int_t^T\tilde{k}(s)
      \e^{-{\textstyle\int_t^s} (\frac{1}{2}k_1(r)+c(r))\vd r}\E^{(x,t)}(|\hat{D}_sF|_s^2)\,\vd s\r);
  \end{align*}
\item for any $F\in \mathscr{F}C_{0,T}^{\infty}$, any continuous
  function $c$ on $[0,T]$, and any $t_1<t_2$ in $[0,T]$,
  \begin{align*}
    \E&\Big[\E[F^2(X_{[0,T]})|\mathscr{F}_{t_2}]\log \E[F^2(X_{[0,T]})|\mathscr{F}_{t_2}]\Big]
        -\E\Big[\E[F^2(X_{[0,T]})|\mathscr{F}_{t_1}]\log \E[F^2(X_{[0,T]})|\mathscr{F}_{t_1}]\Big]\\
      &\qquad\qquad\qquad\qquad\leq 2\int_{t_1}^{t_2}\l(1+\frac{1}{2}\int_t^T\tilde{k}(s)
        \e^{-{\textstyle\int_t^s} (\frac{1}{2} k_1(r)-c(r))\vd r}\vd s\r)\\
      &\qquad\qquad\qquad\qquad\qquad\times\l(\E|\hat{D}_tF|_t^2+\frac{1}{2}\int_t^T\tilde{k}(s)
        \e^{-{\textstyle\int_t^s} (\frac{1}{2} k_1(r)+c(r))\vd r}\E|\hat{D}_sF|_s^2\vd s\r)\vd t;
  \end{align*}
\item for any $F\in \mathscr{F}C_{0,T}^{\infty}$, any continuous
  function $c$ on $[0,T]$, and any $t_1<t_2$ in $[0,T]$,
  \begin{align*}
    \E\Big[\E[F(X_{[0,T]})|\mathscr{F}_{t_2}]^2\Big]&-\E\Big[\E[F(X_{[0,T]})|\mathscr{F}_{t_1}]^2\Big]\\
 &\leq \int_{t_1}^{t_2}\l(1+\frac{1}{2}\int_t^T\tilde{k}(s)\e^{-{\textstyle\int_t^s} (\frac{1}{2} k_1(r)-c(r))\vd r}\vd s\r)\\
 &\quad \ \ \times\l(\E|\hat{D}_tF|_t^2+\frac{1}{2}\int_t^T\tilde{k}(s)\e^{-{\textstyle\int_t^s} (\frac{1}{2} k_1(r)+c(r))\vd r}\E|\hat{D}_sF|_s^2\vd s\r)\vd t.
  \end{align*}
\end{enumerate}
\end{proposition}

\begin{remark}\label{rem1}
  When $k_1=k_2=0$, $c=0$ and $L=\frac{1}{2}\Delta$, the inequalities
  in Proposition \ref{lem2} are proved in \cite{Naber} to characterize
  solutions of the Ricci flow.
\end{remark}

\begin{proof}[Proof of Propostion \ref{lem2}] By \cite[Theorem
  4.3]{Ch1} we know that for $F\in \mathscr{F}C_{0,T}^{\infty}$,
  one has
  \begin{align*}
    \big|\nabla_{x}^t\,\E^{(x,t)} (F(X_{[0,T]}))\big|_t\leq \E^{(x,t)}|\tilde{D}_tF|_t,
  \end{align*}
  and
  \begin{align*}
    \E&\left[\E[F^2(X_{[0,T]})|\mathscr{F}_{t_2}]\log \E[F^2(X_{[0,T]})|\mathscr{F}_{t_2}]\right]\\
      &-\E\left[\E[F^2(X_{[0,T]})|\mathscr{F}_{t_1}]\log \E[F^2(X_{[0,T]})|\mathscr{F}_{t_1}]\right]
        \leq 2\int_{t_1}^{t_2}\E|\tilde{D}_sF|_s^2\,\vd s.
  \end{align*}
  Analogously, by a similar discussion as in the proof of
  \cite[Theorem 4.3]{Ch1}, we have
$$\E\left[\E[F(X_{[0,T]})|\mathscr{F}_{t_2}]^2\right]-\E\left[\E[F(X_{[0,T]})|\mathscr{F}_{t_2}]^2\right]\\
\leq \int_{t_1}^{t_2}\E|\tilde{D}_sF|_s^2\,\vd s.$$
Hence it suffices again to estimate $|\tilde{D}_tF|_t$.

Defining
$$\bar{k}=\frac{k_1+k_2}{2}\quad\text{and}\quad\tilde{k}=\frac{k_2-k_1}{2},$$
 recall
$$\hat{D}_tF(X_{[0,T]})=\sum_{i=1}^n\1_{\{t\leq
  t_i\}}\,\e^{-\frac{1}{2}{\textstyle\int_t^{t_i}}\bar{k}(r)\vd
  r}\,\partrinv t{t_i}\,\nabla^{t_i}_if(X_{t_1},\ldots,X_{t_n}).$$
Then, we have
\begin{align}\label{eq2}
  \tilde{D}_tF&=\hat{D}_tF+\sum_{i=1}^{n}\1_{\{t\leq t_i\}}\,\left(\tilde{Q}_{t,t_i}-\rm id\right)\e^{-\frac{1}{2}{\textstyle\int_t^{t_i}}\bar{k}(r)\vd r}\partrinv t{t_i}\,{\nabla}^{t_i}_i f(X_{t_1},\ldots,X_{t_n})\notag\\
              &=\hat{D}_tF-\frac{1}{2}\int_t^T Q_{t,s}\,\left(\mathscr{R}^Z_{\partr ts}-\bar{k}(s)\id\right)\,\partrinv ts\hat{D}_sF\,\vd s
\end{align}
where
$\tilde{Q}_{t,s}=\e^{\frac{1}{2}{\textstyle\int_t^s}\bar{k}(r)\vd
  r}{Q}_{t,s}$.
Using similar arguments as in the proof of Proposition \ref{lem1}, we
obtain ``$\text{(i)}\Rightarrow\text{(ii)--(v)}$''.

Conversely, to prove ``$\text{(ii)--(v)}\Rightarrow\text{(i)}$'', the
essential part is to prove $\text{(iii)}\Rightarrow\text{(i)}$. The
trick is again to use the test functionals $F(X_{[0,T]})=f(X_t)$ and
$F(X_{[0,T]})=f(X_s)-\frac{1}{2}f(X_t)$. We refer the reader to
\cite{Naber, Ch-Th:2016a} for detailed calculations.
\end{proof}

\begin{proof}[Proof of Theorem \ref{Fol-dd}]For convenience of the reader, we first recall that for $F\in \mathscr{F}C_{0,T}^{\infty}$,
  one has
$$\E[F^2\log F^2]-\E[F^2]\log \E[F^2]\leq 2\int_0^T\E|\tilde{D}_sF|_s^2\,\vd s$$
and
$$\E\big[F-\E[F]\big]^2\leq \int_0^T\E|\tilde{D}_sF|_s^2\,\vd s.$$
Hence it suffices to estimate $\int_0^T\E|\tilde{D}_sF|_s^2\,\vd s$.
Under condition \eqref{curvature-cond}, we obtain the bounds
$$\l|\mathscr{R}_t^Z(X,X)\r|\leq \big(|k_1|\vee |k_2|\big)(t)\,|X|_t^2$$
and
$$\mathscr{R}_t^Z(X,X)\geq k_1(t)|X|_t^2$$
for all $X\in TM$. Then
\begin{align*}
  \tilde{D}_tF&=D_tF+\sum_{i=1}^{n}\l(\int_t^{t_i}\frac{\vd Q_{t,s}}{\,\vd s}\vd s\r)\,\partrinv t{t_i}\,\nabla^{t_i}_i f(X_{t_1},\ldots, X_{t_n})\\
              &=D_tF-\frac{1}{2}\int_t^{T}Q_{t,s}\,\mathscr{R}^Z_{\partr ts}\,\partrinv ts \,D_sF\,\vd s
\end{align*}
which implies that
\begin{align*}
  |\tilde{D}_tF|_t\leq |D_tF|_t+\frac{1}{2}\int_t^T (|k_1|\vee|k_2|)(s) \e^{-\frac{1}{2}{\textstyle\int_t^s}k_1(r)\vd r}|{D}_sF|_s\,\vd s.
\end{align*}
Using a similar argument as in the proof of Theorem \ref{dd}, we arrive at
\begin{align}\label{eq3}
  \int_0^T|\tilde{D}_tF|_t^2\,\vd t\leq \int_0^T\tilde{\Lambda}^c(t,T,k_1,|k_1|\vee |k_2|)|D_tF|_t^2\,\vd t.
\end{align}
On the other hand, by Proposition \ref{lem2}, we have
\begin{align}\label{esti-1}
  |\tilde{D}_tF|_t^2\leq \l(1+\frac{1}{2}\int_t^T\tilde{k}(s)\e^{-{\textstyle\int_t^s} (\frac{1}{2} k_1-c)(r)\vd r}\vd s\r)
  \l(\E|\hat{D}_tF|_t^2+\frac{1}{2}\int_t^T\tilde{k}(s)\e^{-{\textstyle\int_t^s} (\frac{1}{2}k_1+c)(r)\vd r}\E|\hat{D}_sF|_s^2\,\vd s\r).
\end{align}
Moreover, for $|\hat{D}_tF|_t$, it is easy to see that
\begin{align}\label{esti-2}
  |\hat{D}_tF|_t&=\Big|{D}_tF+\sum_{i=1}^{n}\1_{\{t\leq t_i\}}\,\Big(\e^{-\frac{1}{2}{\textstyle\int_t^{t_i}}\bar{k}(r)\vd r}-1\Big)\partrinv t{t_i}\,\nabla^{t_i}_i f\Big|_t\notag\\
                &\leq|{D}_tF|_t+\frac{1}{2}\int_t^T|\bar{k}(s)|\e^{-\frac{1}{2}{\textstyle\int_t^s}\bar{k}(r)\vd r}|{D}_sF|_s\,\vd s.
\end{align}
Combining this with Eq.~\eqref{esti-1}, and using similar arguments as
in the proof of Theorem \ref{dd}, we obtain
$$\int_0^T|\tilde{D}_tF|_t^2\,\vd t\leq \tilde{S}\l(t,T,k_1,\frac{k_2-k_1}{2}\r)\int_0^T\tilde{\Lambda}^c\l(t,T,\frac{k_2+k_1}{2},\frac{|k_2+k_1|}{2}\r)|D_tF|_t^2\,\vd t.$$
From this and by means of Eq.~\eqref{eq3}, the proof is directly completed.
\end{proof}

The following result is a direct consequence of Theorem \ref{Fol-dd}.
\begin{theorem}\label{th3}
  Assume that there exist two continuous functions $k_1$ and $k_2$
  such that
$$k_1(t)\, |X|_t^2\leq\mathscr{R}_t^Z(X,X)\leq k_2(t)\,|X|_t^2,\quad t\in [0,T]$$
for any vector field $X$ on $M$. Then
$$\gap(\mathcal{L})^{-1}\leq \tilde{H}(T,k_1,k_2).
$$
For the special case that $k_1$ and $k_2$ are constants, the
following asymptotics hold as $T\rightarrow 0$:
\begin{itemize}
\item [(i)] for $k_1\geq0$,
 $$\gap(\mathscr{L})^{-1}\leq 1+\frac{k_2}{2}T+\frac{1}{8}\l(k_2^2-\frac{(7k_1+k_2)(k_1+k_2)k_2}{6(3k_1+k_2)}\r)T^2+{\rm o}(T^2);$$
\item [(ii)] for $k_1+k_2\geq 0$ and $k_1<0$,
$$\gap(\mathscr{L})^{-1}\leq 1+\frac{k_2}{2}T+\frac{1}{8}\l(k_2^2+\frac{2k_1^2-k_2^2-5k_1k_2}{6}\r)T^2+{\rm o}(T^2);$$
\item [(iii)] for $k_1+k_2<0$,
$$\gap(\mathscr{L})^{-1}\leq 1-\frac{k_1}{2}T+\frac{1}{8}\l(k_1^2+\frac{3k_1^2+k_2^2}{4}\r)T^2+{\rm o}(T^2).$$
\end{itemize}
\end{theorem}

\proof[Acknowledgements]This work has been supported by Fonds National
de la Recherche Luxembourg (Open project O14/7628746 GEOMREV). The
first named author acknowledges support by NSFC (Grant No.~A011002)
and Zhejiang Provincial Natural Science Foundation of China (Grant
No. LQ16A010009).

\bibliographystyle{amsplain}%

\bibliography{CT16b}

\end{document}